\documentclass{amsart}
\usepackage{a4}
\usepackage{graphics}
\usepackage{color}
\usepackage{amssymb}
\usepackage[all]{xy}
\usepackage{amsmath}
\usepackage{stmaryrd}
\CompileMatrices

\newtheorem{theorem}{Theorem}[section]

\newtheorem{proposition}[theorem]{Proposition}

\theoremstyle{definition}

\newtheorem{example}[theorem]{Example}

\newtheorem{remark}[theorem]{Remark}
\theoremstyle{remark}

\numberwithin{equation}{section}
\def\Chi{{\mathbb X}}

\def\div{{\rm div}}

\def\mal{\! \cdot \!}

\def\rq#1{\widehat{#1}}
\def\t#1{\widetilde{#1}}
\def\b#1{\overline{#1}}
\def\bangle#1{\langle #1 \rangle}

\def\KK{{\mathbb K}}
\def\TT{{\mathbb T}}
\def\ZZ{{\mathbb Z}}

\def\QQ{{\mathbb Q}}
\def\PP{{\mathbb P}}

\def\Cl{\operatorname{Cl}}

\def\Spec{{\rm Spec}}

\def\lin{{\rm lin}}

\begin{document}
\title[Factorially graded rings of complexity one]%
{Factorially graded rings of complexity one}
\author[J.~Hausen]{J\"urgen Hausen} 
\address{Mathematisches Institut, Universit\"at T\"ubingen,
Auf der Morgenstelle 10, 72076 T\"ubingen, Germany}
\email{juergen.hausen@uni-tuebingen.de}
\author[E.~Herppich]{Elaine Herppich} 
\address{Mathematisches Institut, Universit\"at T\"ubingen,
Auf der Morgenstelle 10, 72076 T\"ubingen, Germany}
\email{elaine.herppich@uni-tuebingen.de}

\begin{abstract}
We consider finitely generated normal algebras 
over an algebraically closed field of characteristic 
zero that come with a complexity one grading 
by a finitely generated abelian group such 
that the conditions of a UFD are satisfied for 
homogeneous elements.
Our main results describe these algebras 
in terms of generators and relations.
We apply this to write down explicitly 
the possible Cox rings of normal complete rational 
varieties with a complexity one torus action.
\end{abstract}


\maketitle

\section{Statement of the results}

The subject of this note are finitely generated 
normal algebras $R = \oplus_K R_w$
over some algebraically closed field $\KK$ 
of characteristic zero 
graded by a finitely generated abelian group $K$.
We are interested in the following homogeneous 
version of a unique factorization domain:
$R$ is called {\em factorially (K-)graded\/}
if every homogeneous nonzero nonunit is a 
product of $K$-primes, where a {\em $K$-prime\/}
element is a homogeneous nonzero nonunit $f \in R$ 
with the property that whenever $f$ divides 
a product of homogeneous elements, then 
it divides one of the factors.
For free $K$, the properties
factorial and factorially graded are 
equivalent~\cite{Anders}, but for a $K$ with 
torsion the latter is more general.
Our motivation to study factorially graded algebras
is that the Cox rings of algebraic varieties are
of this type, see for example~\cite{ArDeHaLa}.

We focus on effective $K$-gradings of complexity 
one, i.e., the $w \in K$ with $R_w \ne 0$ 
generate $K$ and $K$ is of rank $\dim(R)-1$.
Moreover, we suppose that the grading is 
pointed in the sense that $R_0 = \KK$ holds.
The case of a free grading group $K$ and hence 
factorial $R$ was treated in~\cite[Section~1]{HaHeSu}.
Here we settle the more general case of factorial 
gradings allowing torsion.
Our results enable us to write down explicitly 
the possible Cox rings of normal complete
rational varieties with a complexity one 
torus action.
This complements~\cite{HaSu}, where the Cox ring 
of a given variety was computed in terms of the 
torus action.

In order to state our results, let us 
fix the notation.
For $r \ge 1$, let 
$A = (a_0, \ldots, a_r)$ 
be a sequence 
of vectors $a_i = (b_i,c_i)$ in $\KK^2$
such that any pair $(a_i,a_k)$ with
$k \ne i$ is linearly independent,
$\mathfrak{n} = (n_0, \ldots, n_r)$ 
a sequence of positive integers
and $L = (l_{ij})$ a family
of positive integers,
where $0 \le i \le r$ and 
$1 \le j \le n_i$.
For every $0 \le i \le r$, define a monomial
$$
T_i^{l_i} 
\  := \
T_{i1}^{l_{i1}} \cdots T_{in_i}^{l_{in_i}}
\ \in \
S
\ := \ 
\KK[T_{ij}; \; 0 \le i \le r, \; 1 \le j \le n_i].
$$
Moreover, for any two indices $0 \le i,j \le r$,
set $\alpha_{ij} :=  \det(a_i,a_j)  =  b_ic_j-b_jc_i$
and for any three indices 
$0 \le i < j < k \le r$ define 
a trinomial
$$
g_{i,j,k} 
\ := \ 
\alpha_{jk}T_i^{l_i} 
\ + \ 
\alpha_{ki}T_j^{l_j} 
\ + \ 
\alpha_{ij}T_k^{l_k}
\ \in \
S.
$$
We define a grading of $S$ by an abelian group $K$ 
such that all the $g_{i,j,k}$ become homogeneous
of the same degree.
For this, consider the free abelian groups
$$ 
F
\ := \
\bigoplus_{i=0}^r \bigoplus_{j=1}^{n_i}  \ZZ \mal f_{ij}
\ \cong \
\ZZ^n,
\qquad \qquad
N \ := \ \ZZ^r,
$$
where we set $n := n_0 + \ldots + n_r$.
Set $l_i := (l_{i1}, \ldots, l_{in_i})$.
Then we have a linear map $P \colon F \to N$
defined by the $r \times n$ matrix 
\begin{eqnarray*}
P
& = & 
\left(
\begin{array}{cccc}
-l_0 & l_1 &   \ldots & 0 
\\
\vdots & \vdots   & \ddots & \vdots
\\
-l_0 & 0 &\ldots  & l_{r} 
\end{array}
\right).
\end{eqnarray*}
Let $P^* \colon M \to E$ be the dual map,
set $K := E / P^*(M)$ and let $Q \colon E \to K$ 
be the projection.
Let $(e_{ij})$ be the dual basis of $(f_{ij})$
and define a $K$-grading on $S$ by 
$\deg(T_{ij}) := Q(e_{ij})$.
Then all $g_{i,j,k}$ are homogeneous 
of the same degree and we 
obtain a $K$-graded factor algebra 
\begin{eqnarray*}
R(A,\mathfrak{n},L)
& := &
\KK[T_{ij}; \; 0 \le i \le r, \; 1 \le j \le n_i] 
\ / \
\bangle{g_{i,i+1,i+2}; \; 0 \le i \le r-2}.
\end{eqnarray*}
We say that the triple $(A,\mathfrak{n},L)$ 
is {\em sincere\/}, if $r \ge 2$ and 
$n_il_{ij} > 1$ for all $i,j$ hold;
this ensures that there exist in fact relations 
$g_{i,j,k}$ and none of these relations 
contains a linear term.
Note that for $r=1$ we obtain the 
diagonal complexity one gradings of the 
polynomial ring~$S$.

\begin{theorem}
\label{thm:main1}
Let $(A,\mathfrak{n},L)$ be any triple as above.
\begin{enumerate}
\item
The algebra $R(A,\mathfrak{n},L)$ 
is factorially $K$-graded;
the $K$-grading is effective, pointed 
and of complexity one.
\item
Suppose that $(A,\mathfrak{n},L)$ is sincere.
Then $R(A,\mathfrak{n},L)$ is factorial
if and only if the group $K$ is torsion free.
\end{enumerate}
\end{theorem}

The second part of this theorem provides  
examples of factorially graded algebras which 
are not factorial;
note that $K$ is torsion free if and only if the 
numbers $l_i := \gcd(l_{i1}, \ldots, l_{in_i})$
are pairwise coprime.

\begin{example}
Let $A$ consist of the vectors 
$(1,0)$, $(1,1)$ and $(0,1)$,
take $\mathfrak{n} = (1,1,1)$ 
and take the family $L$ given by 
$l_{01} = l_{11} = l_{21} = 2$.
Then the matrix 
$$ 
\left(
\begin{array}{rrr}
-2 & 2 & 0
\\
-2 & 0 & 2
\end{array}
\right)
$$
describes the map $P \colon \ZZ^3 \to \ZZ^2$.
Thus the grading group is 
$K = \ZZ \oplus  \ZZ/2\ZZ \oplus  \ZZ/2\ZZ$.
Concretely this grading can be realized as
$$ 
\deg(T_{01}) = (1,\b{0}, \b{0}),
\qquad
\deg(T_{11}) = (1,\b{1}, \b{0}),
\qquad
\deg(T_{21}) = (1,\b{0}, \b{1}).
$$
The associated algebra $R(A,\mathfrak{n},L)$ 
is factorially $K$-graded but not factorial.
It is explicitly given by
\begin{eqnarray*}
R(A,\mathfrak{n},L)
& = & 
\KK[T_{01},T_{11},T_{21}] / \bangle{T_{01}^2-T_{11}^2+T_{21}^2}.
\end{eqnarray*}
\end{example}

Obvious further examples of algebras with an
effective pointed factorial grading are
$R(A,\mathfrak{n},L)[S_1,\ldots,S_m]$,
graded by $K \times \ZZ^m$
via $\deg(S_i) := e_i$, where $e_i \in \ZZ^m$
denotes the $i$-th canonical basis vector.
The methods of~\cite[Section~3]{HaSu} 
apply directly to our situation and show 
that there are no other examples, i.e., 
we arrive at the following.

\begin{theorem}
\label{thm:main2}
Every finitely generated normal $\KK$-algebra 
with an effective, pointed, factorial grading 
of complexity one is isomorphic to some  
$R(A,\mathfrak{n},L)[S_1,\ldots,S_m]$.
\end{theorem}

We turn to Cox rings. Roughly speaking, the Cox ring 
of a complete normal variety $X$ with finitely 
generated divisor class group $\Cl(X)$ is given
as
\begin{eqnarray*}
\mathcal{R}(X)
& := & 
\bigoplus_{\Cl(X)} \Gamma(X,\mathcal{O}_X(D)),
\end{eqnarray*}
see~\cite{ArDeHaLa} for the details of the 
precise definition. 
As mentioned, our aim is to write down all possible 
Cox rings of normal rational complete varieties with 
a complexity one torus action.
They are obtained from
$R(A,\mathfrak{n},L)[S_1,\ldots,S_m]$
by coarsening the $(K \times \ZZ^m)$-grading
as follows.
Let $0 < s < n+m-r$,
consider an integral $s \times n$ matrix $d$,
an integral $s \times m$ matrix $d'$
and the block matrix
\begin{eqnarray*}
\dot P
& = & 
\left( 
\begin{array}{rr}
P & 0 
\\
d & d'  
\end{array}
\right),
\end{eqnarray*}
where $d$ and $d'$ are chosen in such a manner that 
the columns of the matrix $\dot P$ are pairwise 
different primitive vectors in $\ZZ^{r+s}$
which generate $\QQ^{r+s}$ as a cone.
Consider the linear map of lattices 
$\dot P \colon \ddot F \to \dot N$,
where 
$$
\ddot F 
\ := \ 
F \oplus \ZZ f_1 \oplus \ldots \oplus \ZZ f_m,
\qquad \qquad
\dot N 
\ := \ 
\ZZ^{r+s}.
$$
Let $\dot P^* \colon \dot M \to \ddot E$ be the dual 
map and $\dot Q \colon \ddot E \to \dot K$ the projection,
where $\dot K := \ddot E / \dot P^*(\dot M)$.
Denoting by $e_{ij}$, $e_k$ the dual basis to 
$f_{ij}$, $f_k$, we obtain a $\dot K$-grading 
of $R(A,\mathfrak{n},L)[S_1,\ldots,S_m]$ 
by setting
$$ 
\deg(T_{ij}) \ := \ \dot Q(e_{ij}),
\qquad
\deg(S_{k}) \ := \ \dot Q(e_{k}).
$$

\begin{theorem}
\label{thm:main3}
In the above notation, the following holds.
\begin{enumerate}
\item
The $\dot K$-grading of $R(A,\mathfrak{n},L)[S_1,\ldots,S_m]$ 
is effective, pointed and factorial.
Moreover, $T_{ij}$, $S_k$ define pairwise nonassociated 
$\dot K$-prime generators.
\item
The $\dot K$-graded algebra $R(A,\mathfrak{n},L)[S_1,\ldots,S_m]$ is
the Cox ring of a $\QQ$-factorial rational projective 
variety with a complexity one torus action.
\end{enumerate}
\end{theorem}

\begin{theorem}
\label{thm:main4}
Let $X$ be a normal rational complete variety with a 
torus action of complexity one.
Then the Cox ring of $X$ is isomorphic as a graded ring 
to some $R(A,\mathfrak{n},L)[S_1,\ldots,S_m]$ with 
a $\dot K$-grading as constructed above.
\end{theorem}


\section{Proof of the results}

A very first observation lists basic properties 
of the $K$-grading of $R(A,\mathfrak{n},L)$.
We denote by $\TT^n := (\KK^*)^n$ 
the standard $n$-torus.
Moreover, we work with the diagonal action 
of the quasitorus $H := \Spec \, \KK[K]$ 
on $\KK^n$ given by $t \mal z = (\chi_{ij}(t)z_{ij})$, 
where $\chi_{ij} \in \Chi(H)$ is the character 
corresponding to $Q(e_{ij}) = \deg(T_{ij})$.
By definition, this action stabilizes the 
zero set
$$ 
\t{X}
\ := \ 
V(\KK^n; \, g_{i,i+1,i+2}, \, 0 \le i \le r-2)
\ \subseteq \ 
\KK^n.
$$

\begin{proposition}
\label{lem:distpoints}
The algebra $R(A,\mathfrak{n},L)$
is normal, the $K$-grading is 
pointed, effective and of complexity 
one.
\end{proposition}

\begin{proof}
Effectivity of the $K$-grading is 
given by construction, because the degrees
$\deg(T_{ij}) = Q(e_{ij})$ generate~$K$.
From~\cite[Prop.~1.2]{HaHeSu}, we infer 
that $R(A,\mathfrak{n},L)$ is a normal
complete intersection. 
Thus, we have
$$
\dim(R(A,\mathfrak{n},L))
\ = \  
n-r+1
\ = \ 
\dim(\ker(P)) + 1
$$
which means that the $K$-grading is of 
complexity one.
Now consider the action of 
the quasitorus $H :=  \Spec \, \KK[K]$ 
on $\KK^n$ given by the $K$-grading.
Note that $H \subseteq \TT^n$ 
is the kernel of the 
homomorphism of tori
$$ 
\TT^n \ \to \ \TT^r,
\qquad
(t_0,\ldots,t_r) 
\ \mapsto \ 
\left(
\frac{t_1^{l_1}}{t_0^{l_0}}
, \ldots, 
\frac{t_{r}^{l_{r}}}{t_0^{l_0}}
\right).
$$
The set $\t{X} \subseteq \KK^n$ 
of common zeroes of all the $g_{i,i+1,i+2}$
is $H$-invariant and thus it is invariant 
under the 
one-parameter subgroup of $H$ given by 
$$
\KK^* \ \to H,
\qquad
t \ \mapsto \ (t^{\zeta_{ij}}),
\qquad\qquad
\zeta_{ij}
\ := \ 
n_i^{-1}l_{ij}^{-1} \prod_k n_k \prod_{k,m}{l_{km}}.
$$
Since all $\zeta_{ij}$ are positive,
any orbit of this one-parameter subgroup
in $\KK^n$ has the origin in its closure.
Consequently, every $H$-invariant function 
on $\t{X}$ is constant.
This shows $R(A,\mathfrak{n},L)_0 = \KK$. 
\end{proof}

We say that a Weil divisor on $\t{X}$ 
is {\em $H$-prime\/} if it is non-zero, 
has only multiplicities zero or one and 
the group $H$ permutes transitively the 
prime components with multiplicity one.
Note that the divisor $\div(f)$ of a 
homogeneous function $f \in R(A,\mathfrak{n},L)$
on $\t{X}$ is $H$-prime if and only if $f$ is 
$K$-prime~\cite[Prop.~3.2]{Ha2}.
The following is an essential ingredient 
of the proof.

\begin{proposition}
\label{prop:divsprime}
Regard the  variables $T_{ij}$ as regular 
functions on $\t{X}$.
\begin{enumerate}
\item
The divisors of the $T_{ij}$ on $\t{X}$
are $H$-prime and pairwise different.
In particular, the $T_{ij}$ define pairwise 
nonassociated $K$-prime elements in 
$R(A,\mathfrak{n},L)$.
\item
If the ring $R(A,\mathfrak{n},L)$ is factorial
and $n_il_{ij} > 1$ holds,
then the divisor of $T_{ij}$ on $\t{X}$
is even prime.
\end{enumerate}
\end{proposition}

\begin{proof}
For~(i), we exemplarily show that the divisor of
$T_{01}$ is $H$-prime. 
First note that by~\cite[Lemma~1.3]{HaHeSu}
the zero set $V(\t{X}; T_{01})$ 
is described in $\KK^{n}$ by the equations 
\begin{equation}
\label{eqn:diveqns}
T_{01} \ = \ 0,
\qquad 
\alpha_{s+1\,0}T_s^{l_s}+\alpha_{0s}T_{s+1}^{l_{s+1}} \ = \ 0,
\quad
1 \le s \le r-1.
\end{equation}
Let $h \in S$ denote the product of all 
$T_{ij}$ with $(i,j) \ne (0,1)$.
Then, in $\KK^n_h$, the above equations are 
equivalent to 
$$
T_{01} \ = \ 0,
\qquad 
-\frac{\alpha_{s+1\,0}T_s^{l_s}}{\alpha_{0s}T_{s+1}^{l_{s+1}}}
\ = \
1,
\quad
1\leq s\leq r-1.
$$
Now, choose a point $\tilde x \in \KK^n_h$ 
satisfying these equations.
Then $\tilde x_{01}$ is the only vanishing 
coordinate of $\tilde x$.
Any other such point is 
of the form
$$ 
(0,\, t_{02}\tilde x_{02},\, 
\ldots ,\, 
t_{rn_r}\tilde x_{rn_r}),
\qquad
t_{ij} \in \KK^*,
\
t_s^{l_s} = t_{s+1}^{l_{s+1}},
\
1\leq s\leq r-1.
$$
Setting $t_{01} := t_{02}^{-l_{02}}\cdots t_{0n_0}^{-l_{0n_0}}t_1^{l_1}$,
we obtain an element $t = (t_{ij}) \in H$
such that the above point equals $t \mal \tilde x$.
This consideration shows
\begin{eqnarray*}
V(\t{X}_h; T_{01}) & = & H \mal \tilde x.
\end{eqnarray*}
Using~\cite[Lemma~1.4]{HaHeSu}, we see that 
$V(\t{X}; T_{01},T_{ij})$ is of codimension at least two
in $\t{X}$ whenever $(i,j) \ne (0,1)$.
This allows to conclude
\begin{eqnarray*}
V(\t{X}; T_{01}) & = & \b{H \mal \tilde x}.
\end{eqnarray*}
Thus, to obtain that $T_{01}$ defines 
an $H$-prime divisor on $\t{X}$, 
we only need that
the equations~(\ref{eqn:diveqns})
define a radical ideal.
This in turn follows from the 
fact that their Jacobian at the point 
$\tilde x \in V(\t{X};T_{01})$ is of 
full rank.

To verify~(ii), 
let $R(A,\mathfrak{n},L)$ be factorial. 
Assume that the divisor of $T_{ij}$ 
is not prime. 
Then we have $T_{ij} = h_1 \cdots h_s$
with prime elements $h_l \in R(A,\mathfrak{n},L)$.
Consider their decomposition
into homogeneous parts
\begin{eqnarray*}
h_l 
& = &
\sum_{w \in K} h_{l,w}.
\end{eqnarray*}
Plugging this into the product $h_1 \cdots h_s$
shows that $\deg(T_{ij})$ is a 
positive combination of some $\deg(T_{kl})$ 
with $(k,l) \ne (i,j)$.
Thus, there is a vector
$(c_{kl}) \in \ker(Q) \subseteq E$ 
with $c_{ij} = 1$ and $c_{kl} \le 0$ 
whenever $(k,l) \ne (i,j)$.
Since, $\ker(Q)$ is spanned by the rows 
of~$P$, we must have
$n_i = 1$ and $l_{ij}=1$,
a contradiction to our assumptions.
\end{proof}

We come to the main step of the proof, 
the construction of a $T$-variety having 
$R(A,\mathfrak{n},L)[S_1, \ldots, S_m]$ 
as its Cox ring.
We will obtain $X$ as a subvariety of a toric
variety $Z$ and the construction of $Z$ is performed
in terms of fans.
As before, consider the lattice 
\begin{eqnarray*}
\ddot F 
& := & 
F \ \oplus \ \ZZ f_1 \oplus \ldots \oplus \ZZ f_m.
\end{eqnarray*} 
Let $\ddot \Sigma$ be the fan in $\ddot F$
having the rays $\ddot \varrho_{ij}$ 
and~$\ddot \varrho_{k}$ 
through the basis vectors $f_{ij}$ and~$f_{k}$ 
as its maximal cones.
Let $0 < s < n+m-r$,
choose an integral $s \times n$ matrix $d$
and an integral $s \times m$ matrix $d'$.
Consider the (block) 
matrices
$$ 
\left( 
\begin{array}{rr}
P & 0 
\\
d & d'  
\end{array}
\right),
\qquad\qquad
\left( 
\begin{array}{rr}
P & 0 
\end{array}
\right)
$$
and suppose that the columns of the first one
are primitive, pairwise different and 
generate $\dot N_{\QQ}$ as a cone, 
where $\dot N := \ZZ^{r+s}$.
With $N := \ZZ^r$, we have the projection
$B \colon \dot N \to N$ onto the 
first $r$ coordinates and the linear 
maps 
$\dot P \colon \ddot F \to \dot N$ 
and 
$\ddot P \colon \ddot F \to N$
respectively given by the above  
matrices.

Let $\dot \Sigma$ be the fan in $\dot N$ 
with the rays 
$\dot \varrho_{ij} := \dot P (\ddot \varrho_{ij})$
and 
$\dot \varrho_{k} := \dot P (\ddot \varrho_{k})$
as its maximal cones.
The ray $\varrho_{ij}$ through the ${ij}$-th 
column of $P$ is given in terms
of the canonical basis vectors  
$v_1,\ldots,v_{r}$ in  $N = \ZZ^r$
as 
$$
\varrho_{ij} 
\ = \ 
\QQ_{\ge 0} v_i,
\quad
1 \le i \le r,
\qquad 
\varrho_{0j} 
\ = \ 
- \QQ_{\ge 0} (v_1+\ldots+v_r).
$$
For fixed $i$ all $\varrho_{ij}$ are 
equal to each other; we list them 
nevertheless all separately in a system 
of fans $\Sigma$ 
having the zero cone as the common 
gluing data; see~\cite{ACHa} for the 
formal definition of this concept.
Finally, we have the fan $\Delta$ in
$\ZZ^r$ with the rays $\QQ_{\ge 0} v_i$
and $- \QQ_{\ge 0} (v_1+\ldots+v_{r})$
as its maximal cones.

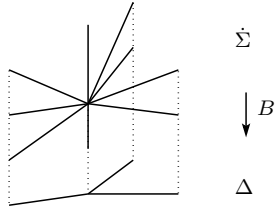
\begin{center}
\begin{picture}(0,0)%
\includegraphics{ambient.pstex}%
\end{picture}%
\setlength{\unitlength}{1243sp}%
\begingroup\makeatletter\ifx\SetFigFont\undefined%
\gdef\SetFigFont#1#2#3#4#5{%
  \reset@font\fontsize{#1}{#2pt}%
  \fontfamily{#3}\fontseries{#4}\fontshape{#5}%
  \selectfont}%
\fi\endgroup%
\begin{picture}(4998,4116)(1543,-6619)
\put(6526,-4786){\makebox(0,0)[lb]{\smash{{\SetFigFont{8}{9.6}{\familydefault}{\mddefault}{\updefault}{\color[rgb]{0,0,0}$B$}%
}}}}
\put(6076,-6361){\makebox(0,0)[lb]{\smash{{\SetFigFont{8}{9.6}{\familydefault}{\mddefault}{\updefault}{\color[rgb]{0,0,0}$\Delta$}%
}}}}
\put(6076,-3436){\makebox(0,0)[lb]{\smash{{\SetFigFont{8}{9.6}{\familydefault}{\mddefault}{\updefault}{\color[rgb]{0,0,0}$\dot \Sigma$}%
}}}}
\end{picture}%

\end{center}

The toric variety $\ddot Z$ associated to 
$\ddot \Sigma$ 
has $\Spec \,  \KK[\ddot E] \cong \TT^{n+m}$
as its acting torus,
where $\ddot E$ is the dual lattice of $\ddot F$.
The fan $\dot \Sigma$ in $\dot N$ 
defines a toric variety $\dot Z$ and the system 
of fans $\Sigma$ defines a toric prevariety $Z$.
The toric prime divisors corresponding to 
the rays 
$\ddot \varrho_{ij}, \ddot \varrho_k \in \ddot \Sigma$,
$\dot \varrho_{ij}, \dot \varrho_k \in \dot \Sigma$
and $\varrho_{ij} \in \Sigma$,
are denoted as 
$$
\ddot D_{ij}, \ddot D_k \subseteq \ddot Z,
\qquad\qquad
\dot D_{ij}, \dot D_k \subseteq \dot Z,
\qquad\qquad
 D_{ij} \subseteq Z.
$$
The toric variety associated to $\Delta$ 
is the open subset $\PP_r^{(1)} \subseteq \PP_r$
of the projective space obtained by removing all 
toric orbits of codimension at least two.
The maps $\dot P$ and $\ddot P$ define toric morphisms
$\dot \pi \colon \ddot Z \to \dot Z$ and 
$\ddot \pi \colon \ddot Z \to Z$.
Moreover, $B \colon \dot N \to N$ 
defines a toric morphism $\beta \colon \dot Z \to Z$
and  the identity $\ZZ^r \to \ZZ^r$ 
defines a toric morphism
$\kappa \colon Z \to \PP_r^{(1)}$. 
These morphisms fit into the commutative diagram
$$ 
\xymatrix{
{\ddot Z}
\ar[rr]^{\dot \pi} 
\ar[dr]_{\ddot \pi}
& & 
{\dot Z}
\ar[dl]^{\beta}
\\
& Z \ar[d]_{\kappa} & 
\\
& {\PP_r^{(1)}} & 
}
$$
Note that $\kappa  \colon Z \to \PP_r^{(1)}$ 
is a local isomorphism which, for fixed $i$, 
identifies all the 
divisors $D_{ij}$ with $1 \le j \le n_i$.
Let $\dot H \subseteq \TT^{n+m}$ 
and 
$\ddot H \subseteq \TT^{n+m}$ 
be the kernels of the toric morphisms
$\dot \pi \colon \ddot Z \to \dot Z$ and 
$\ddot \pi \colon \ddot Z \to Z$
respectively.
Here are the basic features 
of our construction.

\begin{proposition}
\label{prop:ambient}
In the above notation, the following statements 
hold.
\begin{enumerate}
\item 
With $\ddot Z_0 := \ddot Z \setminus 
(\ddot D_1 \cup \ldots \cup \ddot D_m)$,
the restriction $\ddot \pi \colon \ddot Z_0 \to Z$
is a geometric quotient for the action of 
$\ddot H$ on $\ddot Z_0$.
\item
The quasitorus $\dot H$ acts freely on 
$\ddot Z$ and  
$\dot \pi \colon \ddot{Z} \to \dot Z$
is the geometric 
quotient for this action.
\item 
The factor group $G := \ddot H / \dot H$ is 
isomorphic to $\TT^s$ and it acts canonically 
on $\dot Z$.  
\item
The $G$-action on $\dot Z$ has infinite isotropy
groups along $\dot D_1, \ldots ,\dot D_m$ 
and isotropy groups of order $l_{ij}$ along 
$\dot D_{ij}$.
\item 
With $\dot Z_0 := \dot Z \setminus 
(\dot D_1 \cup \ldots \cup \dot D_m)$,
the restriction $\beta \colon \dot Z_0 \to Z$
is a geometric quotient for the action of 
$G$ on $\dot Z$.
\end{enumerate}
\end{proposition}

\begin{proof}
The fact that $\ddot \pi \colon \ddot Z_0 \to Z$
and $\dot \pi \colon \ddot Z \to \dot Z$
are geometric quotients is due to
known characterizations of these notions 
in terms of (systems of) fans, 
see e.g.~\cite{ACHa}.
As a consequence, also $\beta \colon \dot Z_0 \to Z$
is a geometric quotient for the induced action 
of $G = \ddot H / \dot H$.

We verify the remaining part of~(i). 
According to~\cite[Prop.~II.1.4.2]{ArDeHaLa},
the isotropy group of
$\dot H = \ker (\dot \pi)$ 
at a distinguished point
$z_{\ddot \varrho} \in \ddot Z$
has character group isomorphic to
$$ 
\ker(\dot P) \cap \lin_\QQ(\ddot \varrho)
\ \oplus \ 
(\dot P(\lin_\QQ(\ddot \varrho)) \cap \dot N)
/
\dot P(\lin_\QQ(\ddot \varrho) \cap \ddot F).
$$
By the choice of $d$ and $d'$, the map
$\dot P$ sends the primitive generators 
of the rays of $\ddot \Sigma$ 
to the primitive generators of the rays of $\dot \Sigma$.
Thus we obtain that the isotropy 
of $z_{\ddot \varrho_{ij}}$
and $z_{\ddot \varrho_{k}}$ 
are all trivial.

We turn to~(iii).
With the dual lattices $\dot M$ of 
$\dot N$  and $M$ of $N$, we obtain the 
character groups of $\dot H$ and $\ddot H$ 
and the factor group $\ddot H/ \dot H$ as 
$$
\Chi(\dot H) \ = \ \ddot E / \dot P^*(\dot M),
\qquad
\Chi(\ddot H) \ = \ \ddot E / \ddot P^*(M),
\qquad
\Chi(\ddot H/ \dot H) \ = \  \dot P^*(\dot M) /  \ddot P^*(M).
$$
By definition of the matrices of $\dot P$ and $\ddot P$, 
we have $ \dot P^*(\dot M) / \ddot  P^*(M) \cong \ZZ^s$.
This implies $G \cong \TT^s$ as claimed.

To see~(iv), note first that the group 
$G$ equals $\ker(\beta)$ and hence
corresponds to the sublattice 
$\ker(B) \subseteq \ZZ^{r+s}$.
Thus, the isotropy group 
$G_{z_{\dot \varrho}}$ for 
the distinguished point $z_{\dot \varrho} \in \dot Z$ 
corresponding to $\dot \varrho \in \dot \Sigma$ 
has character group isomorphic to
$$ 
\ker(B) \cap \lin_\QQ(\dot \varrho)
\ \oplus \ 
(B(\lin_\QQ(\dot \varrho)) \cap N)
/
B(\lin_\QQ(\dot \varrho) \cap \dot N).
$$
Consequently, for $\dot \varrho = \dot \varrho_{k}$
the isotropy group $G_{z_{\dot \varrho}}$ is infinite
and for $\dot \varrho = \dot \varrho_{ij}$
it is of order $l_{ij}$.
\end{proof}

Now we come to the construction of 
the embedded variety.
Let $\delta \subseteq \ddot F_{\QQ}$ 
be the orthant generated by the basis
vectors $f_{ij}$ and $f_{k}$.
The associated affine toric variety 
$\b{Z} \cong \KK^{n+m}$ is the spectrum
of the polynomial ring
$$
\KK[\ddot E \cap \delta^\vee]
\ = \ 
\KK[T_{ij},S_k; \; 0 \le i \le r, \, 1 \le j \le n_i, \, 1 \le k \le m]
\ = \
S[S_1, \ldots, S_m].
$$
Moreover, $\b{Z}$ contains 
$\ddot Z$ as an open 
toric subvariety and the complement 
$\b{Z} \setminus \ddot Z$ 
is the union of all toric orbits of 
codimension at least two. 
Regarding the trinomials $g_{i,j,k} \in S$ as elements
of the larger polynomial ring $S[S_1, \ldots, S_m]$, 
we obtain an $\ddot H$-invariant subvariety
$$ 
\b{X} 
\ := \ 
V(g_{i,i+1,i+2}; \; 0 \le i \le r-2)
\ \subseteq \ 
\b{Z}.
$$

\begin{proposition}
\label{prop:surfprop}
Set $\ddot{X} := \b{X} \cap \ddot{Z}$.
Consider the images
$\dot X := \dot \pi(\ddot{X}) \subseteq \dot Z$
and $Y := \beta(\dot X) \subseteq Z$.
\begin{enumerate}
\item
$\dot X \subseteq \dot Z$ 
is a normal closed $G$-invariant $s+1$ dimensional 
variety,
$Y \subseteq Z$ is a closed non-separated 
curve and $\kappa(Y) \subseteq \PP_r$ 
is a line.
\item
The intersection $\dot C_{ij} := \dot X \cap \dot D_{ij}$
with the toric divisor $\dot D_{ij} \subseteq \dot Z$
is a single $G$-orbit with isotropy group 
of order~$l_{ij}$. 
\item
The intersection $\dot C_k  := \dot X \cap \dot D_k$
with the toric divisor $\dot D_k \subseteq \dot Z$
is a smooth rational prime divisor consisting of
points with infinite $G$-isotropy.
\item
For every point $x \in \dot X$ not belonging to some 
$\dot C_{ij}$ or to some $\dot C_k$, the isotropy group 
$G_x$ is trivial.
\item 
The variety $\dot X$ satisfies $\Gamma(\dot X,\mathcal{O}) = \KK$,
its divisor class group and Cox ring are given by 
$$
\qquad
\Cl(\dot X) \ \cong \ \dot K,
\qquad\qquad
\mathcal{R}(\dot X) \ \cong \ R(A,\mathfrak{n},L)[S_1, \ldots, S_m].
$$ 
The variables $T_{ij}$ and $S_k$ define pairwise nonassociated 
$\dot K$-prime elements in $R(A,\mathfrak{n},L)[S_1, \ldots, S_m]$.
\item
There is a  $G$-equivariant completion
$\dot X \subseteq X$ with a $\QQ$-factorial 
projective variety $X$ such that 
$\mathcal{R}(X) = \mathcal{R}(\dot X)$
holds. 
\end{enumerate}
\end{proposition}

\begin{proof}
By the definition of $\ddot P$ and $\ddot H$, 
the closed subvariety $\b{X} \subseteq \b{Z}$ is 
invariant under the action of $\ddot H$.
In particular, $\ddot X$ is $\dot H$-invariant
and thus the image $\dot X := \dot \pi(\ddot{X})$
under the quotient map is closed as well.
Moreover, the dimension of  $\dot X$ equals  
$\dim(\ddot X/\dot H) = s+1$.
Analogously we obtain closedness of 
$Y = \ddot \pi (\ddot X)$.
The image $\kappa(Y) = \kappa(\ddot \pi(\ddot X))$ 
is given in $\PP_r$ by the equations 
$$
\alpha_{jk}U_i
\ + \ 
\alpha_{ki}U_j
\ + \ 
\alpha_{ij}U_k
\ = \ 
0
$$
with the variables $U_0, \ldots, U_r$ on $\PP_r$
corresponding to the toric divisors
given by the rays $\QQ_{\ge 0}v_i$ 
and $-\QQ_{\ge 0}(v_0 + \ldots + v_{r-1})$ 
of $\Delta$;
to see this, use that pulling back 
the above equations via 
$\kappa \circ  \ddot \pi$ gives the defining 
equations for $\ddot{X}$.
Consequently $\kappa(Y)$ is a projective line.
This shows~(i).

We turn to~(ii).
According to Proposition~\ref{prop:divsprime},
the intersection  
$\ddot X \cap \ddot D_{ij}$ is 
a single $\ddot H$-orbit.
Since $\dot \pi \colon \ddot X \to \dot X$ 
is a geometric quotient for the 
$\dot H$-action, we conclude that 
$\dot C_{ij} = \dot \pi(\ddot D_{ij})$ 
is a single $G$-orbit.
Moreover, since $\dot H$ acts freely,
the isotropy group of $G = \ddot H / \dot H$
along $\dot C_{ij}$ equals that of $\ddot H$ 
along $\ddot D_{ij}$ which, 
by Proposition~\ref{prop:ambient}~(iv),
is of order $l_{ij}$.

For~(iii) note first that the restrictions 
$\beta \colon \dot D_k \to Z$ and 
are isomorphisms onto the acting torus 
of $Z$.
Moreover, the restricting $\kappa$
gives an isomorphism of the acting 
tori of $Z$ and $\PP_r$.
Consequently, $\beta$ maps $\dot C_k$ 
isomorphically onto the intersection 
of the line $Y$ with the acting torus 
of $\PP_r$. 
Thus, $\dot C_k$ is a smooth rational 
curve.
Proposition~\ref{prop:ambient}~(iv)
ensures that $\dot C_k$ consists of fixed 
points.
Assertion~(iv) is clear.

We prove~(v). 
From Proposition~\ref{lem:distpoints},
we infer $\Gamma(\ddot X,\mathcal{O})^{\dot H} = \KK$ 
which implies $\Gamma(\dot X,\mathcal{O}) = \KK$.
The next step is to establish a surjection
$\dot K \to \Cl(\dot X)$,
where 
$\dot K := \ddot E / \dot P^*(\dot M)$
is the character group of $\dot H$.
Consider the push forward $\dot \pi_*$
from the $\dot H$-invariant Weil divisors
on $\ddot X$ to the Weil divisors on $\dot X$
sending $\ddot D$ to $\dot \pi (\ddot D)$.
For every $\dot w \in \dot K$,
we fix a $\dot w$-homogeneous rational
function $f_{\dot w} \in \KK(\dot X)$
and define a map
$$ 
\mu \colon
\dot K \ \to \ \Cl(\dot X),
\qquad\qquad
\dot w \ \mapsto \ [\dot \pi_* \div(f_{\dot w})].
$$
One directly checks that this does not 
depend on the choice of the  $f_{\dot w}$
and thus is a well defined homomorphism.
In order to see that it is surjective, note
that due to Proposition~\ref{lem:distpoints}, 
we obtain 
$\dot C_{ij}$ as $\dot \pi_* \div(T_{ij})$
and 
$\dot C_k$ as $\dot \pi_* \div(T_k)$.
The claim then follows from the 
observation that removing all
$\dot C_{ij}$ and $\dot C_k$ from $\dot X$
leaves the set 
$\dot X \cap \TT^{r+s}$ 
which is isomorphic to $V \times \TT^r$
with a proper open subset $V \subseteq \kappa(Y)$
and hence has trivial divisor class 
group.

Now~\cite[Theorem~1.3]{HaSu} shows that  
the Cox ring of $\dot X$ is 
$R(A,\mathfrak{n},L)[S_1, \ldots, S_m]$
with the $\Cl(\dot X)$-grading given
by $\deg(T_{ij}) = [\dot C_{ij}]$ 
and $\deg(S_k) = [\dot C_k]$.
Consequently, $R(A,\mathfrak{n},L)[S_1, \ldots, S_m]$
is factorially $\Cl(\dot X)$-graded and thus 
also the finer $\dot K$-grading is factorial.
Since $\dot H$ acts freely on 
$\ddot X$, we can conclude
$\Cl(\dot X) = \dot K$.

Finally, we construct a completion of 
$\dot X \subseteq X$ as wanted in~(vi).
Choose any simplicial projective 
fan $\dot \Sigma'$ 
in $\dot N$ having the same rays as 
$\Sigma$, see~\cite[Corollary~3.8]{OdPa}.
The associated toric variety $\dot Z'$ 
is projective and it is the good 
quotient of an open toric subset 
$\ddot Z' \subseteq \b{Z}$ by the 
action of $\dot H$.
The closure $X$ of $\dot X$ in $\dot Z'$ 
is projective and, as the good quotient 
of the normal variety $\b{X} \cap \ddot Z'$,
it is normal.
By Proposition~\ref{prop:divsprime}, 
the complement $X \setminus \dot X$ is 
of codimension at least two, which gives 
$\mathcal{R}(X) = \mathcal{R}(\dot X)$. 
From~\cite[Cor.~4.13]{Ha2} we infer that 
$X$ is $\QQ$-factorial.
\end{proof}

\begin{remark}
We may realize any given $R(A,\mathfrak{n},L)$ 
as a subring of the Cox ring of a surface:
For every $l_i = (l_{i1}, \ldots, l_{in_i})$
choose a tuple $d_i = (d_{i1}, \ldots, d_{in_i})$
of positive integers with
$\gcd(l_{ij},d_{ij}) = 1$
and 
$d_{i1}/l_{i1} < \ldots < d_{in_i}/l_{in_i}$.
Then take
\begin{eqnarray*}
\dot P
& := &
\left( 
\begin{array}{rrr}
P & 0 & 0
\\
d & 1 & -1 
\end{array}
\right).
\end{eqnarray*}
\end{remark}

We are ready to verify the main results.
As mentioned, Theorem~\ref{thm:main2} follows 
from~\cite[Sec.~3]{HaSu}.
Moreover, the statements of Theorem~\ref{thm:main3}
are contained in Proposition~\ref{prop:surfprop}.

\begin{proof}[Proof of Theorem~\ref{thm:main1}]
According to Proposition~\ref{lem:distpoints}, the algebra 
$R(A,\mathfrak{n},L)$ is normal and the $K$-grading 
is effective, pointed and of complexity one.
By Proposition~\ref{prop:surfprop}, the 
algebra $R(A,\mathfrak{n},L)[S_1,\ldots, S_m]$ is 
a Cox ring and hence it is factorially 
graded, 
see~\cite{Ha2}.
Clearly the graded subring $R(A,\mathfrak{n},L)$
inherits the latter property.

Now suppose that $(A,\mathfrak{n},L)$ is 
a sincere triple.
If $K$ is torsion free, then $K$-factoriality
of $R(A,\mathfrak{n},L)$ implies factoriality, 
see~\cite[Theorem~4.2]{Anders}.
Conversely, if $R(A,\mathfrak{n},L)$ is factorial,
then the generators $T_{ij}$ are prime by 
Proposition~\ref{prop:divsprime}.
From~\cite[Lemma~1.5]{HaHeSu}, we then infer that 
the numbers $\gcd(l_{i1}, \ldots, l_{in_i})$
are pairwise coprime. 
This implies that $P \colon F \to N$ is surjective
and thus $K$ is torsion free.
\end{proof}

\begin{proof}[Proof of Theorem~\ref{thm:main4}]
According to~\cite[Theorem~1.3]{HaSu}, the 
Cox ring $\mathcal{R}(X)$ is isomorphic to 
a ring $R(A,\mathfrak{n},L)[S_1,\ldots,S_m]$
with a grading by $\dot K := \Cl(X)$ such 
that the variables $T_{ij}$ and $S_k$ are 
homogeneous.
In particular, $X$ is the quotient 
by the action of $\dot H = \Spec \, \KK[\dot K]$
on an open subset $\rq{X}$ of 
$$
\b{X}  
\ = \ 
V(g_{i,i+1,i+2}; \; 0 \le i \le r-2)
\ \subseteq \
\b{Z}.
$$
For $r < 2$, the variety $X$ is toric,
we may assume that $T$ acts as a subtorus
of the big torus and the assertion follows
by standard toric geometry.
So, let $r \ge 2$.
By construction, the grading of 
$R(A,\mathfrak{n},L)[S_1,\ldots,S_m]$
by $\b{K} := K \times \ZZ^m$ is 
the finest possible such that all 
variables $T_{ij}$ and $S_k$ are homogeneous.
Consequently, we have exact sequences
of abelian groups fitting into a commutative
diagram
\begin{equation}
\label{snake} 
\xymatrix{
&
&
&
0
\ar[d]
&
\\
&
0
\ar[r]
\ar[d]
&
0
\ar[r]
\ar[d]
&
{\t{K}}
\ar[d]
&
\\
0 
\ar[r]
&
M 
\ar[r]^{\ddot P^*}
\ar[d]
&
{\ddot E}
\ar[r]
\ar@{=}[d]
&
{\b{K}}
\ar[r]
\ar[d]
&
0
\\
0 
\ar[r]
&
{\dot M} 
\ar[r]
\ar[d]
&
{\ddot E}
\ar[r]
\ar[d]
&
{\dot K}
\ar[r]
\ar[d]
&
0
\\
&
{\dot M}/M 
\ar[r]
\ar[d]
&
0
\ar[r]
&
0
&
\\
&
0
&
&
&
}
\end{equation}
In particular we extract from this the following 
two commutative triangles, 
where the second one is obtained by dualizing the 
first one
\begin{equation}
\label{snake2}
\xymatrix{
{\ddot E}
\ar@{<-}[rr]
\ar@{<-}[dr]_{\ddot P^*}
& &
{\dot M}
\ar@{<-}[dl]
\\
& 
M
}
\qquad\qquad
\xymatrix{
{\ddot F}
\ar[rr]^{\dot P}
\ar[dr]_{\ddot P}
& &
{\dot N}
\ar[dl]
\\
& 
N
}
\end{equation}

We claim that the kernel $\t{K}$ is free.
Consider $\b{H} := \Spec \, \KK[\b{K}]$ and the 
isotropy group $\b{H}_{ij} \subseteq \b{H}$ 
of a general point $\rq{x}(i,j) \in \rq{X} \cap V(T_{ij})$. 
Then we have exact sequences
$$
\xymatrix{
1
\ar@{<-}[r]
&
{\b{H} / \b{H}_{ij}} 
\ar@{<-}[r]
&
{\b{H}}
\ar@{<-}[r]
&
{\b{H}}_{ij}
\ar@{<-}[r]
&
1
\\
0 
\ar[r]
&
{\b{K}}(i,j) 
\ar[r]
&
{\b{K}}
\ar[r]
&
{\b{K}}/{\b{K}}(i,j)
\ar[r]
&
0
}
$$
where the second one arises from the first one by 
passing to the character groups. Note that the 
subgroup ${\b{K}}(i,j) \subseteq {\b{K}}$ is given by 
\begin{eqnarray}
\label{eqn:bKij}
{\b{K}}(i,j) 
& = &
\lin_\ZZ(\deg T_{kl}; \; (k,l) \ne (i,j))
+
\lin_\ZZ(\deg T_{p}; \; 1 \le p \le m).
\end{eqnarray}
Now~\cite[Theorem~1.3]{HaSu} tells us that
each variable $T_{ij}$ defines a $\dot K$-prime 
element in $\mathcal{R}(X)$ and thus
its divisor is $\dot H$-prime.
Consequently, $\b{H} / \dot H\b{H}_{ij}$
is connected and has a free 
character group 
$$
\Chi(\b{H} / \dot H \b{H}_{ij})
\ = \ 
\t{K}(i,j) 
\ := \ 
\t{K} \cap {\b{K}}(i,j).
$$
Mimicking equation~(\ref{eqn:bKij}), we 
define a subgroup ${\dot{K}}(i,j) \subseteq {\dot{K}}$
fitting into a commutative net of exact sequences
\begin{equation}
\label{eqn:net}
\xymatrix{
&
0 
\ar[d]
&
0 
\ar[d]
&
0 
\ar[d]
&
\\
0 
\ar[r]
&
{\t{K}}(i,j) 
\ar[r]
\ar[d]
&
{\t{K}}
\ar[r]
\ar[d]
&
{\t{K}}/{\t{K}}(i,j)
\ar[r]
\ar[d]
&
0
\\
0 
\ar[r]
&
{\b{K}}(i,j) 
\ar[r]
\ar[d]
&
{\b{K}}
\ar[r]
\ar[d]
&
{\b{K}}/{\b{K}}(i,j)
\ar[r]
\ar[d]
&
0
\\
0 
\ar[r]
&
{\dot{K}}(i,j) 
\ar[r]
\ar[d]
&
{\dot{K}}
\ar[r]
\ar[d]
&
{\dot{K}}/{\dot{K}}(i,j)
\ar[r]
\ar[d]
&
0
\\
&
0
&
0
&
0
&
}
\end{equation}
By general properties of Cox rings~\cite[Prop.~2.2]{Ha2}, 
we must have ${\dot{K}}/{\dot{K}}(i,j)=0$ and 
thus we can conclude 
\begin{equation}
\label{eqn:netcor}
{\t{K}}/{\t{K}}(i,j) 
\ = \ {\b{K}}/{\b{K}}(i,j) 
\ \cong \
\ZZ/l_{ij}\ZZ.
\end{equation}
Consider again
$\rq{x}(i,j) \in V(T_{ij}) \cap \rq{X}$,
set $x(i,j) = p_X(\rq{x}(i,j))$ and let 
$T$ denote the torus acting on $X$.
Then~\cite[Prop.~2.6]{HaSu} and its proof 
provide us with a commutative diagram
$$ 
\xymatrix{
{\t{H}}
\ar@{<-}[d]
\ar@{}[r]|\supseteq
&
{\t{H}_{ij}}
\ar@{<-}[d]^{\cong}
\\
T
\ar@{}[r]|\supseteq
&
T_{x(i,j)}
}
$$
where $\t{H} = \b{H}/ \dot H$ and 
$\t{H}_{ij} = \b{H} / \dot H\b{H}_{ij}$.
Using~(\ref{eqn:netcor}) and 
passing to the character groups 
we arrive at a commutative diagram
$$
\xymatrix{ 
0 
\ar[r]
&
{\t{K}}(i,j) 
\ar[r]
&
{\t{K}}
\ar[r]
\ar[d]
&
{\ZZ/l_{ij}\ZZ}
\ar[r]
\ar@{<-}[d]^{\cong}
&
0
\\
&
&
{\Chi(T)}
\ar[r]
&
{\ZZ/l_{ij}\ZZ}
\ar[r]
&
0
}
$$
with exact rows.
As seen before, the group 
$\t{K}(i,j)$ is free abelian.
Consequently, also $\t{K}$ must
be free abelian.

Now the snake lemma tells us that $\dot M /M$ is 
free as well.
In particular, the first vertical sequence 
of~(\ref{snake}) splits.
Thus, we obtain the desired matrix presentation of 
$\dot P$ from rewriting the second commutative triangle 
of~(\ref{snake2}) as
$$
\xymatrix{
{\ddot F}
\ar[rr]^{\dot P}
\ar[dr]_{\ddot P}
& &
N \oplus {\dot N} / M^\perp
\ar[dl]
\\
& 
N
}
$$
\end{proof}


\begin{thebibliography}{}%
%
\bibitem{ACHa}
A.~A'Campo-Neuen, J.~Hausen:
Toric prevarieties and subtorus actions.  
Geom. Dedicata  87  (2001),  no. 1-3, 35--64.
%
\bibitem{ArDeHaLa} 
I.~Arzhantsev, U.~Derenthal, J.~Hausen, A.~Laface:
Cox rings, arXiv:1003.4229, see also the authors' 
webpages.
%
\bibitem{Anders} 
D.F.~Anderson:
Graded Krull domains.
Comm. Algebra 7 (1979), no.~1, 79--106.
%
\bibitem{Ha2}
J.~Hausen:
Cox rings and combinatorics II.
Mosc. Math. J. 8 (2008), no.~4,  711--757.
%
\bibitem{HaHeSu}
J.~Hausen, E.~Herppich, H.~S\"u{\ss}:
Multigraded factorial rings and Fano varieties
with torus action.
Documenta Math. 16 (2011), 71--109.
%
\bibitem{HaSu}
J.~Hausen, H.~S\"u{\ss}:
The Cox ring of an algebraic variety 
with torus action.
Advances Math. 225 (2010), 977--1012.
%
\bibitem{OdPa}
T.~Oda, H.S.~Park: 
Linear Gale transforms and Gelfand-Kapranov-Zelevinskij 
decompositions.  
Tohoku Math. J. (2) 43 (1991), no.~3, 375--399.
\end{thebibliography}
\end{document}